\documentclass[12pt]{amsart}

\DeclareFontEncoding{OT2}{}{}

\numberwithin{equation}{section}

\title[On $\Z^d$-odometers associated to integer matrices]{On $\Z^d$-odometers associated to integer matrices}

\author[S. Merenkov]{Sergei Merenkov}
\address{Department of Mathematics, City College of New York and CUNY Graduate Center, New York, NY 10031, USA}
\email{smerenkov@ccny.cuny.edu}
\thanks{S. Merenkov supported by NSF grant DMS-1800180.}
\thanks{M. Sabitova supported by PSC-CUNY Award TRADB-53-92 jointly funded by The Professional Staff Congress and The City University of New York}

\author[M. Sabitova]{Maria Sabitova}
\address{Department of Mathematics, Queens College and CUNY Graduate Center, Flushing, NY 11367, USA}
\email{Maria.Sabitova@qc.cuny.edu}

\usepackage{amsmath, amssymb, amsthm,latexsym}
\usepackage{enumerate}
\usepackage{graphicx}
\usepackage{mathtools}
\usepackage{amsmath}
\usepackage{extarrows}
\usepackage[colorinlistoftodos]{todonotes}

\newcommand\N{{\mathbb N}}

\newcommand\Z{{\mathbb Z}}
\newcommand\Q{{\mathbb Q}}
\newcommand\R{{\mathbb R}}
\newcommand\T{{\mathbb T}}

\newcommand\G{{\mathcal G}}
\newcommand\cS{{\mathcal S}}
\newcommand\cO{{\mathcal O}}

\newcommand\id{\operatorname{id}}

\renewcommand\:{\colon}

\newcommand\la{\lambda}
\newcommand\La{\Lambda}

\newcommand\cP{\mathcal{P}}

\newcommand{\sbr}{\smallbreak}

\newcommand{\M}{\operatorname{M}}
\newcommand{\GL}{\operatorname{GL}}
\newcommand{\bbe}{\begin{equation}}
\newcommand{\ee}{\end{equation}}

\newtheorem{theorem}{Theorem}[section]

\newtheorem{lemma}[theorem]{Lemma}
\newtheorem{definition}[theorem]{Definition}

\newtheorem{proposition}[theorem]{Proposition}
\newtheorem{corollary}[theorem]{Corollary}

\theoremstyle{remark}
\newtheorem{remark}[theorem]{Remark}
\newtheorem{example}[theorem]{Example}

\theoremstyle{definition}

\usepackage{hyperref}
\hypersetup{
    colorlinks=true,
    linkcolor=blue,
    filecolor=magenta,      
    urlcolor=cyan,
}

\begin{document}
\abstract
We extend the results of~\cite{GPS19} on characterization of conjugacy, isomorphism, and continuous orbit equivalence of $\Z^d$-odometers to dimensions $d>2$. We then apply these extensions to the case of odometers defined by matrices with integer coefficients. 
\endabstract
\maketitle

\section{Introduction}\label{s:intro}
\noindent
In this note we first supply an argument that extends the main result of the paper  ``$\Z^d$-odometers and cohomology" by T.~Giordano, I.~F.~Putnam, and C.~F.~Skau~\cite[Theorem~1.5 (1), (2), (3)]{GPS19} to dimensions greater than 2. More precisely, the proofs of conjugacy, isomorphism, and continuous orbit equivalence characterizations of~\cite[Theorem~1.5]{GPS19} are based on~\cite[Theorem~4.4]{GPS19}, the only result in that paper where the dimension restriction $d=1,2$ is important. Our Proposition~\ref{P:Tau} below lifts this restriction and therefore leads to the characterizations in arbitrary dimension $d\ge1$. In turn, our proof of Proposition~\ref{P:Tau} uses Lemma~\ref{L:goodBasis} that made the relevant computations of the image of the first cohomology group under a natural map 
simple and possible in higher dimensions.   
In the last section we apply these results along with the earlier results by the second author~\cite{s1} in the setting of odometers defined by matrices with integer coefficients. 
Below we use the notations and terminology from~\cite{GPS19}, \cite{s}, and \cite{s1} and refer to these papers for more details.

\sbr

If $d\in\N$ and $G$ is a subgroup of $\Z^d$, the group $\Z^d$ acts on $\Z^d/G$ by 
\begin{equation}\label{eq:act}
\phi_G^k(l+G)=k+l+G,\quad k, l\in \Z^d.
\ee 
Let $\G=\{G_1, G_2,\dots\}$ be a decreasing sequence of subgroups of $\Z^d$ of finite index, {\it i.e.}, 
$$
\Z^d=G_1\supseteq G_2\supseteq\cdots,\quad [\Z^d\: G_n]<\infty,\ n\in\N.
$$
For $n\in\N$ and $\Z^d\supseteq G_n\supseteq G_{n+1}$, let $q_n\:\Z^d/G_{n+1}\to\Z^d/G_n$ denote the quotient map $q_n(k+G_{n+1})=k+G_n,\ k\in\Z^d$.  
\begin{definition}
Let $\G=\{G_1, G_2,\dots\}$ be as above and let $X_{\G}$ denote the inverse limit system
$$
\Z^d/G_1\xlongleftarrow{q_1}\Z^d/G_2\xlongleftarrow{q_2}\cdots.
$$ 
A $\Z^d$-\emph{odometer} is a pair $(X_{\G},\phi_\G)$, where $\phi_\G$ is the natural action of $\Z^d$ on 
$X_{\G}$ induced by $\phi_{G_n}^k$, $k\in\Z^d$, $n\in\N$.
The natural projection from $X_{\G}$ to $\Z^d/G_n$ is denoted by $\pi_n$.
\end{definition}

If $G_n\neq G_{n+1}$ for infinitely many $n\in\N$, then $X_\G$ is a Cantor set and $\phi_\G$ is a minimal action of $\Z^d$ on it, which is free if and only if $\cap_{n=1}^\infty G_n=\{0\}$. The action $\phi_\G$ is also isometric in the metric $d_\G$ given by 
$$
d_\G(x,y)=\sup\{0,1/n\,\,\vert\,\,\pi_n(x)\neq\pi_n(y)\},\quad x,y\in X_\G.
$$
Moreover, $X_\G$ supports a unique $\phi_\G$-invariant probability measure $\mu_\G$ such that
$$
\mu_\G\left(\pi_n^{-1}(k+G_n)\right)=\frac{1}{[\Z^d\: G_n]},\quad n\in\N,\ k\in \Z^d.
$$

\sbr

In~\cite{GPS19} the authors proposed a system equivalent to $(X_\G, \phi_\G)$ in the sense of conjugacy defined below, using Pontryagin duality. 
Namely, let $H$ be a subgroup of $\Q^d$ so that $H$ contains $\Z^d$. Let $Y_H=\widehat{H/\Z^d}$ be the Pontryagin dual of the quotient. Here, the groups $H$ and $H/\Z^d$ are endowed with the discrete topology, and hence $Y_H$ is compact. 
Let $\T^d=\R^d/\Z^d$ denote the $d$-torus and let $\rho:H/\Z^d\rightarrow\T^d$ be the map induced by the inclusion $\Q^d\hookrightarrow \R^d$, {\it i.e.}, $\rho\left((r_1,\dots, r_d)+\Z^d\right)=\left(e^{2\pi ir_1},\dots, e^{2\pi ir_d}\right)$,
$(r_1,\dots, r_d)\in H$. Identifying the Pontryagin dual of $\T^d$ with $\Z^d$, we have the dual map
$\hat\rho:\Z^d\rightarrow Y_H$. Then the action $\psi_H$ of $\Z^d$ on $Y_H$ is defined via 
$$
\psi_H^k(x)=x+\hat\rho(k),\quad k\in\Z^d,\ x\in Y_H.
$$ 
The action $(Y_H,\psi_H)$ is free if and only if $H$ is dense in $\Q^d$.
\begin{definition}
Let $X,Y$ be topological spaces, let $G$ be a group, and let $\phi$, $\psi$ be actions of $G$ on $X$, $Y$, respectively, by homeomorphisms. 
The actions $(X,\phi)$ and $(Y,\psi)$ of $G$ are said to be \emph{conjugate} if there exists a homeomorphism $h\: X\to Y$ such that 
$$
h\circ \phi^g=\psi^g\circ h,
$$  
for all $g\in G$. In this case, we refer to $h$ as a \emph{conjugacy} between the two actions.  
\end{definition}
If $H$ is an increasing union of finite index extensions $H_n$ of $\Z^d$, $n\in\N$, then, up to conjugacy, $(Y_H, \psi_H)$ is  the inverse limit of
\bbe\label{eq:lim}
(Y_{H_1}, \psi_{H_1})\xlongleftarrow{\widehat{i_1}}(Y_{H_2}, \psi_{H_2})\xlongleftarrow{\widehat{i_2}}\cdots,
\ee 
where $i_n\: H_n/\Z^d\to H_{n+1}/\Z^d$ is the inclusion, $n\in\N$.
The correspondence between $\Z^d$-odometers and $\Z^d$-actions $(Y_H,\psi_H)$, up to conjugacy, is established by realizing $(Y_H,\psi_H)$ as an inverse limit \eqref{eq:lim} and by passing to dual lattices; see~\cite[Theorems~2.5, 2.6]{GPS19}. Here, if $K$  is a lattice in $\R^d$, its dual lattice $K^*$ is
$$
K^*=\{x\in\R^d \,\,\vert\,\,\langle k,x\rangle\in\Z\ {\rm for\ all}\ k\in K\}.
$$

We now consider odometers associated to integer matrices.
For a non-singular $d\times d$-matrix $A$ with integer coefficients, $A\in\operatorname{M}_d(\Z)$, define
$$
G_A=\left.\left\{A^{-k}{x}\,\,\right\vert\, {x}\in\Z^d,\,k\in\Z \right\},\quad
\Z^d\subseteq G_A\subseteq \Q^d.
$$
One can readily check that $G_A$ is a subgroup of $\Q^d$.
Applying the process described above to the group $H=G_A$, we get an {associated} $\Z^d$-odometer $Y_{G_A}$, defined up to conjugacy. 
In \cite{s}, the second author classified groups $G_A$ in the $2$-dimensional case and applied the results to $2$-dimensional odometers $Y_{G_A}$ using 
\cite[Theorem 1.5]{GPS19}. In \cite{s1},
 it was studied when $G_A$, $G_B$ are isomorphic as abstract groups for non-singular $A,B\in\M_d(\Z)$ for an arbitrary $d$. We combine the results from \cite{s1} with the results of this paper to analyze when $Y_{G_A}$, $Y_{G_B}$ are equivalent with respect to conjugacy, isomorphism, continuously orbit equivalence, and orbit equivalence.

\sbr
\noindent
\textbf{Acknowledgements.} The authors thank Gleb Aminov for useful discussions. We also thank the anonymous referee for careful reading and valuable comments.

\section{Conjugacy and Isomorphism}
\noindent
In what follows, $\operatorname{M}_d(\Z)$ denotes the ring of $d\times d$-matrices with integer coefficients and
$\operatorname{GL}_d(\Z)$ denotes the group of non-singular matrices $A\in\operatorname{M}_d(\Z)$ with $\det A=\pm 1$.
Proposition~\ref{P:Tau} below extends~\cite[Theorem~4.4]{GPS19} to higher dimensions and its proof requires the following lemma.

\begin{lemma}\label{L:goodBasis}
Let $d\in\N$ and let $G$ be a subgroup of $\Z^d$ of finite index. 
Then for each $i=1,\dots, d$, there exists a free basis $\mathcal B=\{f_1,\dots, f_d\}$ of $G$ such that $f_i=a_i e_i$, where $a_i\in\N$ and $e_i$ is the $i$-th vector of the standard basis of $\Z^d$.  Equivalently, for any non-singular $M\in\operatorname{M}_d(\Z)$ and each $i=1,\dots, d$, there exists 
$P(i)\in\operatorname{GL}_d(\Z)$ such that the $i$-th column of $MP(i)$ is $a_ie_i$ for some $a_i\in\N$.
\end{lemma}
\begin{proof}
Since $G$ is a subgroup of $\Z^d$ of finite index, $G$ is a free group of rank $d$. Thus, the columns of $M$ form a basis of $G$.
We first prove the following:
\begin{lemma}\label{L:small}
For any $j$ there exists a basis $g_1,\dots, g_d$ of $G$ such that the $j$-th component $[g_l]_j$
of $g_l$ is zero for any $l=2,\ldots, d$ and $[g_1]_j=a_j$, $a_j\in\N$. Equivalently, the $j$-th row of the matrix $\left(\begin{matrix}
g_1 & g_2 & \ldots & g_d
\end{matrix}
\right)$ $($we write coordinates of each $g_i$ as a column$)$ is $\left(\begin{matrix}
a_j & 0 & \ldots & 0
\end{matrix}
\right)$.
\end{lemma}
\begin{proof}[Proof of Lemma \ref{L:small}]
Let $g_1,\dots, g_d$ be an arbitrary basis of $G$ and
let 
$$
M=\left(\begin{matrix}
g_1 & g_2 & \ldots & g_d
\end{matrix}
\right)=(g_{ik}), 
$$
$
M\in\operatorname{M}_d(\Z)$ is non-singular. Note that there exists at least one non-zero element in the $j$-th row of $M$, since
$\det M\ne 0$. Assume there are two non-zero elements in the $j$-th row of $M$. Without loss of generality (we can interchange columns of $M$), $g_{j1}\ne 0$, $g_{j2}\ne 0$. By multiplying $g_1$, $g_2$ by $-1$ if necessary, we can assume $g_{j1}> 0$, $g_{j2}> 0$. If $g_{j1}=g_{j2}$, then 
 $g'_1,\dots, g'_d$ is a new basis of $G$ with $g'_2=g_2-g_1$, $g'_l=g_l$, $l\ne 2$, and $[g'_2]_j=0$.
 If $g_{j1}\ne g_{j2}$, then without loss of generality, we can assume $g_{j1}> g_{j2}$.
 Let $a_1=g_{j1}$, $a_2=g_{j2}$, and $a_1=a_2 k+a_3$, for some $k,a_3\in\Z$, $0\leq a_3<a_2$. Then 
 $g'_1,\dots, g'_d$ is a new basis of $G$ with $g'_1=g_2$, $g'_2=g_1-kg_2$, $g'_l=g_l$, $l\ne 1,2$, and 
 $[g'_1]_j=a_2$, $[g'_2]_j=a_3$, $0\leq a_3<a_2<a_1$. Continuing this way, we get a sequence
 $0\leq\cdots<a_h<\cdots< a_3<a_2<a_1$ of non-negative integer numbers, which in finitely many steps has to reach zero (this is essentially the Euclidean algorithm). This shows that there exists a basis
 $g'_1,\dots, g'_d$ of $G$ such that $[g'_2]_j=0$. Repeating the process for any other
 $g'_s$, $g'_t$  with $s\ne t$, $[g'_s]_j\ne 0$, $[g'_t]_j\ne 0$, we conclude that there exists a basis 
 $g''_1,\dots, g''_d$ of $G$ such that $[g''_1]_j\in\N$, $[g'_l]_j=0$ for any $l\ne 1$.
\end{proof}
We now use induction on $d$ to prove Lemma \ref{L:goodBasis}. Let $d=1$. Then $e_1=1$, $G=a\Z$ for some $a>0$, $f_1=ae_1$, and the claim follows. Let $d>1$. We consider two cases: $i>1$ and $i=1$. Let $i>1$. By Lemma \ref{L:small} applied to $j=1$, there exists a basis $g_1,\dots, g_d$ of $G$ such that
$$
M=\left(\begin{matrix}
g_1 & \ldots & g_d
\end{matrix}
\right)=\left(\begin{matrix}
a_1 & 0  \\
* & M'
\end{matrix}
\right),\quad\quad a_1\in\N,\,\,M'\in \operatorname{M}_{d-1}(\Z),\,\,\det M'\ne 0.
$$
By induction on $d$, there exists $S\in\operatorname{GL}_{d-1}(\Z)$ such that 
the $(i-1)$-st column of $M'S$ is $be'_{i-1}$, where $b\in\N$ and $e'_{i-1}$ is the
 $(i-1)$-st vector of the standard basis of $\Z^{d-1}$. 
 Let 
$$
P=\left(\begin{matrix}
1 & 0  \\
0 & S
\end{matrix}
\right)\in\operatorname{GL}_{d}(\Z).
$$
Then the $i$-th column of $MP$
is $be_i$ and the claim follows for $i>1$.

\sbr

Let $i=1$. By Lemma \ref{L:small} applied to $j=d$, there exists a basis $g_1,\dots, g_d$ of $G$ such that
$$
M=\left(\begin{matrix}
g_d & \ldots & g_1
\end{matrix}
\right)=\left(\begin{matrix}
M'' & *  \\
0 & a_d
\end{matrix}
\right),\quad\quad a_d\in\N,\,\,M''\in \operatorname{M}_{d-1}(\Z),\,\,\det M''\ne 0.
$$
As in the case $i>1$, we apply the induction on $d$ to $M''$. Thus, there exists
$S'\in\operatorname{GL}_{d-1}(\Z)$ such that 
the $1$-st column of $M''S'$ is $b'e'_{1}$, where $b'\in\N$ and $e'_{1}$ is the
 $1$-st vector of the standard basis of $\Z^{d-1}$. Let 
$$
P'=\left(\begin{matrix}
S' & 0  \\
0 & 1
\end{matrix}
\right)\in\operatorname{GL}_{d}(\Z).
$$
Then the $1$-st column of $MP'$
is $b'e_1$ 
and the claim follows in the case $i=1$ as well.
\end{proof}

For a system $(X,\phi)$, where $X$ is a topological space and $\phi$ is an action of $\Z^d$ on $X$ by homeomorphisms, the first cohomology group $H^1(X,\phi)$ is defined as follows. A 1-cocycle $\theta$ is a continuous function $\theta\: X\times\Z^d\to\Z$ such that
$$
\theta(x,m+n)=\theta(x,m)+\theta(\phi^m(x), n),\quad x\in X,\ m,n\in\Z^d.
$$
A 1-cocycle $\theta$ is a coboundary if and only if there exists a continuous function $h\: X\to\Z$ such that 
$$
\theta(x,n)=h\left(\phi^n(x)\right)-h(x),\quad x\in X,\ n\in\Z^d.
$$
Let $Z^1(X,\phi)$ and $B^1(X,\phi)$ denote the groups of 1-cocycles and coboundaries, respectively, and let $H^1(X,\phi)=Z^1(X,\phi)/B^1(X,\phi)$ be the first cohomology group.  

\sbr

We now recall the definition of the map $\tau_\mu^1$, where $\mu$ is an invariant probability measure on
a $\Z^d$-action $(X,\phi)$. If $\theta$ is a 1-cocycle, 
$\tau_\mu^1(\theta)\in{\rm Hom}(\Z^d,\R)$ is given by
$$
\tau_\mu^1(\theta)(n)=\int_X \theta(x,n)\, d\mu(x),\quad n\in \Z^d.
$$  
Since $\tau_\mu^1(\theta)=0$ if $\theta$ is a coboundary, $\tau_\mu^1$ passes to a well-defined group homomorphism
\bbe\label{eq:uh1}
\tau_\mu^1\: H^1(X,\phi)\to {\rm Hom}(\Z^d,\R).
\ee
The space ${\rm Hom}(\Z^d,\R)$ is identified with $\R^d$ via the map that takes $\alpha\in{\rm Hom}(\Z^d,\R)$ to $(\alpha(e_1),\dots, \alpha(e_d))\in \R^d$.   

\sbr

We also denote by $B_\mu^1(X,\phi)$ the group of 1-cocycles $\theta(x,n)\in Z^1(X,\phi)$ such that 
$$
\tau_\mu^1(\theta)=0,
$$
and by $H_\mu^1(X,\phi)$ the group $Z^1(X,\phi)/B_\mu^1(X,\phi)$. The group $H_\mu^1(X,\phi)$ is a quotient of $H^1(X,\phi)$. By definition of $B_\mu^1(X,\phi)$, the map $\tau_\mu^1$ factors through the quotient $H_\mu^1(X,\phi)$, and, by abuse of notation, we denote the resulting map also by $\tau_\mu^1$:
\bbe\label{eq:uh2}
\tau_\mu^1\: H^1_{\mu}(X,\phi)\to {\rm Hom}(\Z^d,\R).
\ee

\begin{proposition}\label{P:Tau}
Let $d\in\N$ and let $H$ be a dense subgroup of $\Q^d$ such that $\Z^d\subseteq H$. Let $\mu$ be the unique invariant probability measure for  $\Z^d$-action $(Y_H, \psi_H)$. Then the map
$$
\tau_\mu^1\: H_\mu^1(Y_H,\psi_H)\to H
$$
given by \eqref{eq:uh2} is an isomorphism.
\end{proposition}
\begin{proof}
We supplement the proof of~\cite[Theorem~4.4]{GPS19} with Lemma~\ref{L:goodBasis} above. This allows to simplify computations and also extends \cite[Theorem~4.4]{GPS19} to an arbitrary dimension.

\sbr

It is immediate from the definitions above that $\tau_\mu^1$ is an injective group homomorphism. 
We write $H=\cup_{n\in\N} H_n$, the union of an increasing sequence of finite index extensions $H_n$ of $\Z^d$. Let $G_n=H_n^*$ be the dual lattice. Then $\mathcal G=\{G_1, G_2,\dots\}$ is a decresing sequence of finite index subgroups of $\Z^d$ such that $(Y_H,\psi_H)$ and $(X_{\mathcal G},\phi_{\mathcal G})$ are conjugate.  

\sbr

The 
group $H^1(Y_H,\psi_H)$ is the direct limit of the sequence $H^1(\Z^d/G_n,\phi_{G_n})$, where $\phi_{G_n}$ is the natural action of $\Z^d$ on $\Z^d/G_n$ given by \eqref{eq:act}, $n\in\N$. Therefore, if $[\theta]\in H_\mu^1(Y_H,\psi_H)$, then $\theta$ is a cocycle in $Z^1(\Z^d/G_n,\phi_{G_n})$ for some $n$. We can write
\begin{equation}\label{E:Form}
\tau_\mu^1(\theta)=[\Z^d\: G_n]^{-1}\sum_{k\in F} \left(\theta(k+G_n, e_1), \dots,\theta(k+G_n, e_d)\right),
\end{equation}
where $F$ is a fundamental domain for $G_n$.  Since $\theta(k+G_n, e_i)\in\Z$ for any $i=1,\dots, d$, we have
$\tau_\mu^1(\theta)\in H_n$ and hence $\tau_\mu^1(\theta)\in H$. 

\sbr

Clearly, to show that $\tau_\mu^1\: H_\mu^1(Y_H,\psi_H)\to H$ given by \eqref{eq:uh2} is surjective, it is enough  to show that $\tau_\mu^1\: H^1(Y_H,\psi_H)\to H$ given by \eqref{eq:uh1}  is surjective.
Let $h\in H$. Then $h\in H_n$ for some $n\in\N$. We show that there exists $[\theta]\in  H^1(Y_H,\psi_H)$ such that $\tau_\mu^1(\theta)=h$.
According to~\cite[Lemma~4.2]{GPS19}, for each $\theta\in Z^1(\Z^d/G_n,\phi_{G_n})$, the map $\alpha(\theta)$ given by
$$
\alpha(\theta)(g)=\theta(G_n,g),\quad g\in G_n,
$$
induces an isomorphism
$\alpha\: H^1(\Z^d/G_n,\phi_{G_n})\to {\rm Hom}(G_n,\Z)$. The group ${\rm Hom}(G_n,\Z)$ is identified with $H_n$ via the canonical inner product $\langle\cdot\,,\cdot\rangle$ on $\R^d$, and thus  there exists $\theta\in Z^1(\Z^d/G_n,\phi_{G_n})$ with $\alpha(\theta)=\langle\cdot\,,h\rangle$. 

\sbr

It remains to show $\tau_\mu^1(\theta)=h$, {\it i.e.},
\begin{equation}\label{E:Form2}
\sum_{k\in F} \theta(k+G_n, e_i)=[\Z^d\: G_n]\langle e_i, h\rangle\quad {\rm for\ each}\ i=1,\dots, d.
\end{equation}
For each fixed $i=1,\dots,d$, we apply Lemma~\ref{L:goodBasis} to find a basis $\mathcal B=\{f_1,\dots, f_d\}$ for $G_n$ such that $f_i=a e_i$, where $a\in\N$. 
We choose a fundamental domain $F$ for $G_n$ in Equation~\eqref{E:Form2} to consist of elements from $\Z^d$ inside a parallelepiped in $\R^d$ determined by the basis $\mathcal B$. Namely,
$$
F=\{t_1f_1+\dots +t_d f_d\,\,\vert\,\,0\le t_1,\dots, t_d<1\}\cap \Z^d.
$$ 

The map $\theta$ is defined by 
$$
\theta(k_1+G_n, k_2+g)=\langle g',h\rangle,
$$
for $k_1, k_2\in F, g\in G_n$, where $g'\in G_n$ is the unique element such that $k_1+k_2+g=k'+g'$ and $k'\in F$. 
One can check that so defined, $\theta$ is a 1-cocycle.
To compute the left-hand side of~\eqref{E:Form2}, we choose $k_1=k\in F$, $k_2= e_i$, and $g=0$. Note that $k_2\in F$, because of our choice of $F$. If $k\in F$ is such that $k+ e_i$ is also in $F$, then $g'=0$, and thus the term $\theta(k+G_n, e_i)$ does not contribute anything to the sum in~\eqref{E:Form2}.

\sbr

Now assume $k\in F$ is such that $k+ e_i$ is not in $F$. Let $F_i$ denote the set of all such $k$. 
In this case $g'=f_i$, because $k'=k+ e_i-f_i\in F$. Indeed, let $P_{k,i}$ consist of all elements $n\in F$ such that for each $j=1,\dots, d,\ j\neq i$, the $j$-th component of $n$ coincides with the $j$-th component of $k$. Since $F$ is a parallelepiped with side $f_i=a e_i,\ a\in\N$, the $i$-th coordinates of elements $n\in P_{k,i}$ form a sequence of $a$ consecutive integers $m, m+1,\dots, m+a-1$.  
Now, by the choice of $k$,  its $i$-th component is $m+a-1$. For $j=1,\dots, d,\ j\neq i$, the $j$-th component of $k'$ equals the $j$-th component of $k$. The $i$-th component of $k'$ equals $m$, and thus $k'\in P_{k,i}\subseteq F$. 

\sbr

We conclude that $\theta(k+G_n, e_i)=\langle f_i, h\rangle=a\langle e_i, h\rangle$ for each $k\in F_i$. 
The number of elements $k$ in $F_i$ equals the number of elements in the projection $P_i$ of $F$ to $\Z^{d-1}$ obtained from $\Z^d$ by omitting the $i$-th coordinate. 
Indeed, by omitting the $i$-th coordinate of $k\in F_i\subseteq F$, we get an element in $P_i$. Conversely, for any $k_i\in P_i$, the $i$-th coordinates of all elements in $F$ that project to $k_i$ form a sequence $m, m+1,\dots, m+a-1$, since $F$ is a parallelepiped with side $f_i= a e_i,\ a\in\N$. The element $k$ that projects to $k_i$ and whose $i$-th coordinate is $m+a-1$ is in $F_i$.   
The projection $P_i$ is itself a parallelepiped in $\Z^{d-1}$. Let $b$ denote the number of elements in $P_i$, which is the same as the number of elements in $F_i$. We therefore conclude 
$$
\sum_{k\in F}\theta(k+G_n, e_i)=\sum_{k\in F_i}\theta(k+G_n, e_i)=ab\langle e_i, h\rangle.
$$
Note that $ab$ is the number of elements in $F$, which is equal to $[\Z^d\: G_n]$. Thus, \eqref{E:Form2} holds.
\end{proof}

A characterization of a $\Z^d$-action $(Y_H,\psi_H)$ up to conjugacy now follows from Proposition~\ref{P:Tau} and the following elementary lemma.

\begin{lemma}\label{L:Conj}
Let $d\in\N$. 
If $\Z^d$-actions $(Y_{H_1},\psi_{H_1})$ and $(Y_{H_2},\psi_{H_2})$ are conjugate, then 
$$
\tau_{\mu_1}^1\left(H_{\mu_1}^1(Y_{H_1},\psi_{H_1})\right)
=\tau_{\mu_2}^1\left(H_{\mu_2}^1(Y_{H_2},\psi_{H_2})\right),
$$
where $\mu_1, \mu_2$ are the unique probability measures for $(Y_{H_1},\psi_{H_1}), (Y_{H_2},\psi_{H_2})$, respectively.
\end{lemma}
\begin{proof}
Let $h$ be a conjugacy from $(Y_{H_1},\psi_{H_1})$ to $(Y_{H_2},\psi_{H_2})$. It induces an isomorphism 
$$
h^*\: H_{\mu_2}^1(Y_{H_2},\psi_{H_2})\to H_{\mu_1}^1(Y_{H_1},\psi_{H_1})
$$ 
given by $h^*([\theta])=[h^*(\theta)]$, where $h^*(\theta(y,n))=\theta(h(x), n)$, $y=h(x)$, $x\in Y_{H_1}$. 
The invariance of $\mu_1$ and uniqueness of $\mu_2$ imply $h_*(\mu_1)=\mu_2$, where $h_*(\mu_1)$ denotes the push-forward measure of $\mu_1$ under $h$.
Therefore, for $[\theta(y, n)]\in H_{\mu_2}^1(Y_{H_2},\psi_{H_2})$, one has
$$
\begin{aligned}
\tau_{\mu_2}^1(\theta)(n)&=\int_{Y_{H_2}} \theta(y,n)\, d\mu_2=\int_{Y_{H_2}} \theta(y,n)\, dh_*(\mu_1)
\\
&=\int_{Y_{H_1}} \theta(h(x),n)\, d\mu_1=\tau_{\mu_1}^1(h^*(\theta))(n).
\end{aligned}
$$ 
\end{proof}

\begin{corollary}\label{C:Conj}
Let $H_1,H_2$ be dense subgroups of $\Q^d$ such that $\Z^d\subseteq H_1$, $\Z^d\subseteq H_2$. 
Two $\Z^{d}$-actions $(Y_{H_1}, \psi_{H_1})$ and $(Y_{H_2}, \psi_{H_2})$ are conjugate if and only if   $H_1=H_2$.
\end{corollary}

\begin{proof}
If $H_1=H_2$, then the conjugacy is trivial. Assume $(Y_{H_1}, \psi_{H_1})$ and $(Y_{H_2}, \psi_{H_2})$ are conjugate. By Proposition~\ref{P:Tau},  
$$
\tau_{\mu_1}^1\left(H_{\mu_1}^1(Y_{H_1}, \psi_{H_1})\right)=H_1, \quad  \tau_{\mu_2}^1\left(H^1_{\mu_2}(Y_{H_2}, \psi_{H_2})\right)=H_2,
$$
where $\mu_1$ and $\mu_2$ are the unique invariant probability measures for $(Y_{H_1}, \psi_{H_1})$ and $(Y_{H_2}, \psi_{H_2})$, respectively. By Lemma~\ref{L:Conj}, we have $H_1=H_2$.
\end{proof}

\begin{definition}
Let $(X,\phi)$ be an action of a group $G$ and let $(Y,\psi)$ be an action of a group $H$. An \emph{isomorphism} between the actions is a pair $(h,\alpha)$, where $h\: X\to Y$ is a homeomorphism and $\alpha\: G\to H$ is a group isomorphism, such that
$$
h\circ \phi^g=\psi^{\alpha(g)}\circ h.
$$
If such a pair $(h,\alpha)$ exists, then $(X,\phi)$ and $(Y,\psi)$ are said to be \emph{isomorphic}.
\end{definition}

\begin{proposition}\label{P:Iso}
Let $d_1,d_2\in\N$ and let $H_1$ $($resp., $H_2)$ be a dense subgroup of $\Q^{d_1}$ $($resp., $\Q^{d_2})$ that contains 
$\Z^{d_1}$ $($resp., $\Z^{d_2})$.  
A $\Z^{d_1}$-action $(Y_{H_1}, \psi_{H_1})$ is isomorphic to a $\Z^{d_2}$-action $(Y_{H_2}, \psi_{H_2})$ if and only if $d_1=d_2$, the common value being denoted by $d$, and there exists $A\in {\rm GL}_d(\Z)$ such that $A H_1=H_2$.
\end{proposition}
\begin{proof}
If  $d=d_1=d_2$ and $AH_1=H_2$ for some $A\in {\rm GL}_d(\Z)$, then $(Y_{A H_1}, \psi_{A H_1})$ is trivially conjugate to $(Y_{H_2}, \psi_{H_2})$. Moreover, the actions $(Y_{A H_1}, \psi_{A H_1})$ and $(Y_{H_1}, \psi_{H_1})$ are isomorphic via the automorphism of $\Z^d$ defined by $A$ (see~\cite[Proposition~2.8]{GPS19}).

\sbr

Conversely, assume $(Y_{H_1}, \psi_{H_1})$ and  $(Y_{H_2}, \psi_{H_2})$ are isomorphic. The existence of a group isomorphism $\alpha\: H_1\to H_2$ implies $d_1=d_2$. Indeed, $d_i$ is the rank of $H_i,\ i=1,2$, and a group isomorphism preserves the rank. 
We denote the common value of $d_1, d_2$ by $d$. By 
\cite[Proposition~2.8]{GPS19}, there is $A\in {\rm GL}_d(\Z)$ such that $(Y_{A H_1}, \psi_{A H_1})$ is conjugate to $(Y_{H_2}, \psi_{H_2})$. We now apply Corollary~\ref{C:Conj} to conclude $A H_1=H_2$. 
\end{proof}

\section{Continuous Orbit Equivalence}
\noindent
\begin{definition}
An action $(X,\phi)$ of a group $G$ and an action $(Y,\psi)$ of another group $G'$ are said to be \emph{orbit equivalent} if there exists a homeomorphism $h\: X\to Y$ such that for each $x\in X$ one has
\bbe\label{eq:h}
h\left(\{\phi^g(x)\,\,\vert\,\, g\in G\}\right)=\{\psi^{g'}(h(x))\,\,\vert\,\, g'\in G'\}.
\ee
\end{definition}

In~\cite{GPS19}, the authors characterize orbit equivalence of $\Z^d$-action $(Y_H, \psi_H)$, where $H$ is a dense subgroup of $\Q^d$ containing $\Z^d$, using superindex $[[H\:\Z^d]]$. The \emph{superindex} is defined as 
$$
[[H\:\Z^d]]=\left\{[H'\: \Z^d]\,\,\vert\,\,\Z^d\subseteq H'\subseteq H,\ [H'\:\Z^d]<\infty\right\}.
$$

\begin{theorem}\cite[Corollary~5.5]{GPS19}\label{T:OrbEq}
Let $d,d'\in\N$, let $H$ be a dense subgroup of $\Q^d$ that contains $\Z^d$, and let $H'$ be a dense subgroup of $\Q^{d'}$ that contains $\Z^{d'}$. The $\Z^d$-action $(Y_H, \psi_H)$ and the $\Z^{d'}$-action $(Y_{H'}, \psi_{H'})$ are orbit equivalent if and only if 
$$[[H\:\Z^d]]=[[H'\:\Z^{d'}]].$$ 
\end{theorem} 

Let $(X,\phi)$ be an action of a group $G$ and let $(Y,\psi)$ be an action of a group $G'$. Assume $(X,\phi)$ and $(Y,\psi)$ are orbit equivalent. By definition, there exists a homeomorphism $h:X\rightarrow Y$ satisfying \eqref{eq:h}. If the actions 
$(X,\phi)$, $(Y,\psi)$ are free, then there exist unique maps $\alpha\: X\times G\to G'$ and $\beta\: Y\times G'\to G$ such that
$$
h(\phi^g(x))=\psi^{\alpha(x,g)}(h(x)),
$$
for all $x\in X, g\in G$, and also
$$
h^{-1}(\psi^{g'}(y)=\phi^{\beta(y,g')}(h^{-1}(y)),
$$
for all $y\in Y, g'\in G'$. The maps $\alpha$, $\beta$ are called  \emph{orbit cocycles}.

\begin{definition}
Let $(X,\phi)$ and $(Y,\psi)$ be free actions of groups $G$ and $G'$, respectively, such that $(X,\phi)$ and $(Y,\psi)$ are orbit equivalent. Then the actions $(X,\phi)$ and $(Y,\psi)$ are called \emph{continuously orbit equivalent} if there is a homeomorphism $h\: X\to Y$ such that the associated orbit cocycles $\alpha$ and $\beta$ are continuous in the corresponding product topologies.
\end{definition}

\begin{proposition}\label{P:ContOrbEqv} 
Let $d_1,d_2\in\N$, let $(Y_{H_1}, \psi_{H_1})$ be a free $\Z^{d_1}$-action, and let $(Y_{H_2},\psi_{H_2})$ be a free $\Z^{d_2}$-action. Then $(Y_{H_1}, \psi_{H_1})$ and $(Y_{H_2},\psi_{H_2})$ are continuously orbit equivalent if and only if $d=d_1=d_2$ and there exists $A\in {\rm GL}_d(\Q)$ with $\det A=\pm 1$ and $AH_1=H_2$.
\end{proposition}
\begin{proof}
Here ${\rm GL}_d(\Q)$ denotes the group of non-singular $d\times d$-matrices with rational coefficients. 
The proof follows the lines of the proof of~\cite[Theorem~5.7]{GPS19}, where one uses Proposition~\ref{P:Iso} in place of~\cite[Corollary~5.1]{GPS19}.
\end{proof}

\begin{remark}
Investigation of continuous orbit equivalence in the general setting was carried out in \cite{li}. In particular, conditions on when continuous orbit equivalence implies isomorphic equivalence were given in that paper.
\end{remark}

\section{Odometers defined by matrices}
\noindent
In this section we generalize the results in \cite{s} on $\Z^2$-odometers defined by matrices with integer coefficients to the $d$-dimensional case, $d>2$. 
For convenience, we first put together the results from previous sections in one theorem. To simplify notation, 
we denote a $\Z^d$-action $(Y_H,\psi_H)$ defined above for a subgroup $H$ of $\Q^d$ containing $\Z^d$ 
by $Y_H$. 

\begin{theorem}\label{th:chug}
Let $d\in\N$ and let $H_1,H_2$ be dense 
subgroups of $\Q^d$ such that 
$\Z^d\subseteq H_1$, $\Z^d\subseteq H_2$. Then
\begin{enumerate}[$(1)$]
\item $\Z^d$-actions $Y_{H_1}$, $Y_{H_2}$ are conjugate if and only if $H_1=H_2$.
\item $\Z^d$-actions $Y_{H_1}$, $Y_{H_2}$ are isomorphic if and only if there is $T\in\GL_d(\Z)$ such that $TH_1=H_2$.
\item $\Z^d$-actions $Y_{H_1}$, $Y_{H_2}$ are continuously orbit equivalent if and only if there is $T\in\GL_d(\Q)$ such that $\det T=\pm 1$ and $TH_1=H_2$.
\item Assume $H_1$, $H_2$ are dense in $\Q^d$. Then $\Z^d$-actions $Y_{H_1}$, $Y_{H_2}$ are orbit equivalent if and only if 
$[[H_1:\Z^d]]=[[ H_2:\Z^d]]$.
\end{enumerate}
\end{theorem} 
\begin{proof}
This is the content of Corollary \ref{C:Conj}, Proposition \ref{P:Iso}, Theorem \ref{T:OrbEq}, and Proposition \ref{P:ContOrbEqv} above.
\end{proof}

Recall that for a non-singular $d\times d$-matrix $A$ with integer coefficients, $A\in\operatorname{M}_d(\Z)$,
$$
G_A=\left.\left\{A^{-k}{ x}\,\right\vert\, { x}\in\Z^d,\,k\in\Z \right\},\quad
\Z^d\subseteq G_A\subseteq \Q^d.
$$
One can check that $G_A$ is a subgroup of $\Q^d$. 
For simplicity, we denote $Y_{G_A}$ by $Y_{A}$. Note that $G_A$ is naturally isomorphic 
to the inductive limit of the system $(\Z^d,f_j)_{j\in\N}$, where each $f_j:\Z^d\rightarrow \Z^d$ is given by multiplication by $A$, 
$f_j({\bf x})=A{\bf x}$, ${\bf x}\in\Z^d$, $j\in\N$.

\sbr

We start with a characterization of dense groups $G_A$ in $\Q^d$. 
Let $h_A\in\Z[t]$ be the characteristic polynomial of $A$ and 
let $h_A=h_1h_2\cdots h_s$,
where $h_1,h_2,\ldots,h_s\in\Z[t]$ are non-constant and irreducible.

\begin{lemma}\cite[Lemma 8.1]{s1}\label{l:dense}
The group $G_A$ is dense in $\Q^d$ if and only if $h_i(0)\ne\pm 1$ for all $i\in\{1,2,\ldots,s\}$.
\end{lemma}

Next, we describe orbit equivalent odometers.

\begin{lemma}\cite[Lemma 8.2]{s1}\label{l:orbit}
Let $A,B\in \operatorname{M}_d(\Z)$ be non-singular such that 
$G_A$ $($resp., $G_B$$)$ is dense in $\Q^d$.
Then $\Z^d$-actions $Y_A$, $Y_B$ are orbit equivalent if and only if $\det A,\det B$ have the same prime divisors $($in $\Z)$.
\end{lemma}

The next lemma is a special (simple) case when all the equivalences hold at the same time. 

\begin{lemma}\label{lem:ttodo}
Let $A,B\in\operatorname{M}_d(\Z)$ be non-singular such that $G_A$ $($resp., $G_B$$)$ is dense in 
$\Q^d$. Assume that for any prime $p\in\N$ that divides $\det A$ we have 
$$
h_A\equiv t^d\,(\text{mod }p).
$$
Then the following are equivalent:
\begin{enumerate}[$(1)$]
\item $\Z^d$-actions $Y_A$, $Y_B$ are conjugate;
\item $\Z^d$-actions $Y_A$, $Y_B$ are isomorphic;
\item $\Z^d$-actions $Y_A$, $Y_B$ are continuously orbit equivalent;
\item $\det A$, $\det B$ have the same prime divisors and for any prime $p\in\N$ that divides $\det B$ we have
$$
h_B\equiv t^d\,(\text{mod }p).
$$
\end{enumerate}
\end{lemma}
\begin{proof}
Follows from Theorem \ref{th:chug} and \cite[Lemma 3.10]{s}.
\end{proof}

Let
$$
\det A=ap_1^{s_1}p_2^{s_2}\cdots p_l^{s_l}
$$
be the prime-power factorization of $\det A$, where $p_1,p_2,\ldots,p_l\in\N$ are distinct primes, $a=\pm 1$, and $s_1,s_2,\ldots,s_l\in\N$. Let 
$$
\cP=\cP(A)=\{p_1,p_2,\ldots,p_l\}.
$$
If $\cP=\emptyset$, equivalently, $\det A=\pm 1$, then $G_A=\Z^d$ and $Y_A$ is trivial. 
Moreover, 
 for a non-singular $B\in\M_d(\Z)$ we know that $TG_A=G_B$ for some
$T\in\GL_d(\Q)$  if and only if $\det B=\pm 1$ and hence 
$G_B=\Z^d$ \cite[Lemma 3.2(i)]{s}.

\sbr
 
In what follows we assume $\cP\ne\emptyset$.
Denote
$$
\cP'=\cP'(A)=\left\{p\in\cP,\,\,h_A\not\equiv t^d\,(\text{mod }p)\right\},
$$
where $h_A\in\Z[t]$ denotes the characteristic polynomial of $A$. The case 
$\cP'=\emptyset$ is settled in Lemma \ref{lem:ttodo}, so in what follows we assume $\cP'\ne\emptyset$. Finally, for a prime $p\in\N$ 
let $t_p=t_p(A)$ denote the multiplicity of zero in the reduction of the characteristic polynomial of $A$ modulo $p$, 
$0\leq t_p\leq d$. Thus, in our notation $t_p\ne d$ if and only if $p\in\cP'$.

\sbr

Even though the results in \cite{s1} apply to isomorphisms between groups $G_A$, $G_B$ for arbitrary non-singular $A,B\in\M_d(\Q)$, to avoid making this paper too technical, we only consider a generic case when the characteristic polynomials of $A, B$ are irreducible. An interested reader could use \cite{s1} together with Theorem \ref{th:chug} to treat other cases. Another additional assumption we make is the condition that there exists $p\in\cP'$ such that the greatest common divisor $(t_p,d)$ of $t_p$ and $d$ is $1$. It seems that this condition provides the right setting for
the generalization of the $2$-dimensional case to higher dimensions. In particular, in the $2$-dimensional case, if $h_A$ is irreducible and $TG_A=G_B$ for some $T\in\GL_d(\Q)$, then $h_B$ is irreducible and $T$ takes an eigenvector of $A$ to an 
eigenvector of $B$ \cite{s}. It turns out these facts remain true under the assumption $(t_p,d)=1$ and not true in general, {\em e.g.}, all $t_p=2$ and $d=4$ \cite{s1}.  

\sbr

Let $\overline{\Q}$ denote a fixed algebraic closure of $\Q$. Recall that $\overline{\Q}$ consists of algebraic numbers, roots of non-zero polynomials in one variable with rational coefficients. The eigenvalues of a matrix $A$ with integer coefficients are algebraic numbers since they are roots of the characteristic polynomial of $A$.
Let $A,B\in\M_d(\Z)$ be non-singular and let $\la_1,\ldots,\la_d\in\overline{\Q}$ 
(resp., $\mu_1,\ldots,\mu_d\in\overline{\Q}$) denote eigenvalues of $A$ (resp., $B$). Assume there exists $T\in\GL_d(\Q)$ that satisfies $TG_A=G_B$. Suppose further that the characteristic polynomial $h_A\in\Z[t]$ of $A$ is irreducible and there exists a prime 
$p\in\cP(A)$ such that $(t_p(A),d)=1$. The following facts follow from \cite[Proposition 5.7]{s1}. Namely, $h_B\in\Z[t]$ is irreducible. Moreover, (up to rearrangement of eigenvalues)
\begin{eqnarray}
&& \Q(\la_1)=\Q(\mu_1)\text{ and }\la_1,\,\mu_1 
\text{ have the same}\label{eq:112233}\\
&& \text{prime divisors in the ring of integers.}\nonumber 
\end{eqnarray}
In particular, the splitting fields of $h_A$, $h_B$ coincide. 
Furthermore, both $A$ and $B$ are diagonalizable (over $\overline{\Q}$) and there exist $M,N\in\GL_d(\overline{\Q})$ such that
\bbe\label{eq:matr}
A=M\left(\begin{matrix}
\la_1 & \cdots & 0 \\
\vdots & \ddots & \vdots \\
0 & \cdots & \la_d 
\end{matrix}
\right)M^{-1},\,\,
B=N\left(\begin{matrix}
\mu_1 & \cdots & 0 \\
\vdots & \ddots & \vdots \\
0 & \cdots & \mu_d
\end{matrix}
\right)N^{-1},
\ee
and $T=NM^{-1}$. 
If $\Z^d$-actions $Y_A$, $Y_B$ are conjugate (or isomorphic, or continuously orbit equivalent), then by Theorem \ref{th:chug}, we have $T=I_d$, the identity matrix (or $T\in\GL_d(\Z)$, or $\det T=\pm 1$, respectively) 
and, hence,  $M=N$ (or $NM^{-1}\in\GL_d(\Z)$, or $\det NM^{-1}=\pm 1$, respectively). This proves the following lemma.

\begin{lemma}\label{l:irredodo}
Let $A,B\in\M_d(\Z)$ be non-singular such that $G_A$ $($resp., $G_B$$)$ is dense in $\Q^d$. 
Assume the characteristic polynomial $h_A\in\Z[t]$ of $A$ is irreducible and $(t_p,d)=1$ for some prime 
$p\in\cP$. 
\begin{enumerate}
\item[$(i)$] If $\Z^d$-actions $Y_A$, $Y_B$ are conjugate, then \eqref{eq:112233} holds, there exist $M,N\in\GL_d(\overline{\Q})$ such that 
\eqref{eq:matr} holds, and $M=N$.
\item[$(ii)$] If $\Z^d$-actions $Y_A$, $Y_B$ are isomorphic, then \eqref{eq:112233} holds, there exist $M,N\in\GL_d(\overline{\Q})$ such that 
\eqref{eq:matr} holds, and $NM^{-1}\in\GL_d(\Z)$.
\item[$(iii)$] If $\Z^d$-actions $Y_A$, $Y_B$ are continuously orbit equivalent, then \eqref{eq:112233} holds, there exist $M,N\in\GL_d(\overline{\Q})$ such that 
\eqref{eq:matr} holds, and $\det NM^{-1}=\pm 1$.
\end{enumerate}
%
\end{lemma}

As in the $2$-dimensional case, the conditions in Lemma \ref{l:irredodo} are also sufficient in the cases of conjugacy and isomorphism. 

\begin{lemma}\label{l:conjs}
Let $A,B\in\M_d(\Z)$ be non-singular such that $G_A$ $($resp., $G_B$$)$ is dense in $\Q^d$. 
Assume $h_A\in\Z[t]$ is irreducible and $(t_p,d)=1$ for some prime 
$p\in\cP$.
\begin{enumerate}
\item[$(i)$] If \eqref{eq:112233} holds and there exists $M\in\GL_d(\overline{\Q})$ such that 
\eqref{eq:matr} holds for $N=M$, then $\Z^d$-actions $Y_A$, $Y_B$ are conjugate.
\item[$(ii)$] 
 If \eqref{eq:112233} holds, there exist $M,N\in\GL_d(\overline{\Q})$ such that $NM^{-1}\in\GL_d(\Z)$ and
\eqref{eq:matr} holds, then $\Z^d$-actions $Y_A$, $Y_B$ are isomorphic.
\end{enumerate}
\end{lemma}
\begin{proof}
First, $(ii)$ follows easily from $(i)$ as in the proof of \cite[Lemma 8.10]{s}. We repeat the argument for the sake of completeness. Assume $(i)$ holds, let $X=NM^{-1}\in\GL_d(\Z)$, and assume \eqref{eq:112233} and 
\eqref{eq:matr} hold. Then $XM=N$ and
$$
XG_A=G_{XAX^{-1}}=G_{N\La N^{-1}}=G_B,\quad \La=\left(\begin{matrix}
\la_1 & \cdots & 0 \\
\vdots & \ddots & \vdots \\
0 & \cdots & \la_d 
\end{matrix}
\right).
$$
Here, $XAX^{-1}=N\La N^{-1}\in\M_d(\Z)$, since $X\in\GL_d(\Z)$. Also, $G_{N\La N^{-1}}=G_B$ by $(i)$ and Theorem \ref{th:chug}. Thus, $\Z^d$-actions $Y_A$, $Y_B$ are isomorphic by Theorem \ref{th:chug}.

\sbr

We now prove $(i)$. Assume \eqref{eq:112233}, 
\eqref{eq:matr} hold and $N=M$. Since $h_A$ is irreducible, \eqref{eq:112233} implies that $h_B$ is also irreducible.
Indeed, $h_B$ is irreducible if and only if $[\Q(\mu_1):\Q]=d$. From \eqref{eq:matr} with $N=M$, we see that $A$, $B$ share the same eigenvector ${\bf u}\in\overline{\Q}^d$ such that $A{\bf u}=\la_1{\bf u}$, $B{\bf u}=\mu_1{\bf u}$. Since $h_A$ is irreducible, the Galois group $G=\operatorname{Gal}(\overline{\Q}/\Q)$ acts transitively on the eigenvalues of $A$, {\it i.e.}, there exist $\sigma_2,\ldots,\sigma_d\in G$ such that $\la_i=\sigma_i(\la_1)$, $\sigma_1=\id$, $1\leq i\leq d$. Since $A$, $B$ have integer coefficients, by applying $\sigma_i$ to $A{\bf u}=\la_1{\bf u}$, $B{\bf u}=\mu_1{\bf u}$, we conclude that 
$\sigma_i({\bf u})$ is an eigenvector of $A$ (resp., $B$) corresponding to $\la_i$ (resp., $\sigma_i(\mu_1)$), $1\leq i\leq d$, and $\sigma_1(\mu_1),\ldots, \sigma_d(\mu_1)$ are all (distinct) eigenvalues of $B$. 
By abuse of notation, let $\mu_i=\sigma_i(\mu_i)$, $1\leq i\leq d$. Moreover, it follows from  \eqref{eq:112233} that $h_A$, $h_B$ share the same splitting field $K$ and each pair $\la_i$, $\mu_i$ share the same prime ideal divisors in the ring of integers $\cO_K$ of $K$. Therefore, $\cP=\cP(A)=\cP(B)$, $\cP'=\cP'(A)=\cP'(B)$, and $t_p=t_p(A)=t_p(B)$ for any prime $p\in\N$. 
In \cite{s1}, we discuss the characteristic $\{\alpha_{pij} \}_{p,i,j}$ of a group $G_A$ with respect to a free basis $\{{\bf f}_1,\ldots,{\bf f}_d\}$ of $\Z^d$, where $p\in\cP'(A)$, $\alpha_{pij}\in\Z_p$, $1\leq i\leq t_p(A)$, $t_p(A)+1\leq j\leq d$. The system
$\cS(A)=\{{\bf f}_1,\ldots,{\bf f}_d,\alpha_{pij}\,\,\vert\,\,p,i,j\}$ determines all generators of $G_A$ over $\Z$ \cite[Lemma 3.5]{s1}. Furthermore, we show that $\cS(A)$ can be calculated from eigenvectors of $A$ corresponding to eigenvalues divisible by a prime ideal of $\cO_K$ that divides $p$, $p\in\cP'$ \cite[Remark 4.4]{s1}. Since each pair 
$\la_i$, $\mu_i$ share the same prime ideal divisors and $A$, $B$ share the same eigenvectors corresponding to 
$\la_i$, $\mu_i$, respectively, we see that $G_A$, $G_B$ share the same set of generators. Therefore, $G_A=G_B$.
\end{proof}

\begin{remark}
It follows from its proof that in Lemma \ref{l:conjs}$(i)$, instead of \eqref{eq:matr} with $N=M$, it is enough to assume that $A$, $B$ share the same eigenvector ${\bf u}\in\overline{\Q}^d$ such that $A{\bf u}=\la_1{\bf u}$, $B{\bf u}=\mu_1{\bf u}$. 

\sbr

It is well-known (it follows from the Latimer--MacDuffee--Taussky Theorem) that for a fixed monic irreducible polynomial $h\in\Z[t]$ of degree $d$ there are finitely many $\GL_d(\Z)$-conjugacy classes $[A]$ of $A\in\M_d(\Z)$ with characteristic polynomial $h$. Let $n=n(h)$ denote the number of conjugacy classes. Also, for $S\in\GL_d(\Z)$ we have $SG_A=G_{SAS^{-1}}$. As was discussed above, we know that if there exists $t_p$ with $(t_p,d)=1$, $A,B\in\M_d(\Z)$ share the same irreducible characteristic polynomial and $TG_A=G_B$ for some $T\in\GL_d(\Q)$, then $TAT^{-1}=B$ \cite{s1}. This implies that there are exactly $n$ isomorphism classes 
$[Y_A]_{isom}$ of odometers $Y_A$, where $A\in\M_d(\Z)$ has characteristic polynomial $h$, and there are less or equal than $n$
continuously orbit equivalent classes $[Y_A]_{cont.orb}$ of odometers $Y_A$, where $A\in\M_d(\Z)$ has characteristic polynomial $h$. 
\end{remark}

\begin{remark}
Note that $M=N$ implies $A$, $B$ commute. However, $AB=BA$ does not imply conjugacy, since the ordering of eigenvalues matters. For example, $G_A\ne G_B$ for
$
A=\left(\begin{matrix}
3 & 0 \\
0 & 5 
\end{matrix}
\right),\,\,
B=\left(\begin{matrix}
5 & 0 \\
0 & 3 
\end{matrix}
\right).
$
\end{remark}

\begin{remark}
We know from the $2$-dimensional case that for odometers $Y_A$ defined by  $A\in\M_d(\Z)$, continuous orbit equivalence is more subtle than conjugacy and isomorphism. 
Even in the $2$-dimensional case, general sufficient conditions for $\Z^2$-actions $Y_A$, $Y_B$ to be 
continuously orbit equivalent under the conditions of Lemma \ref{l:irredodo} become rather technical. 
However, in each particular example, it is possible to resolve the question using Theorem \ref{th:chug} and the results 
in \cite{s} and \cite{s1}. 
\end{remark}

\begin{example}\label{ex:7} 
In this example, $d=3$ and
$A,B, C\in\M_3(\Z)$ have the same irreducible characteristic polynomial $h(t)=t^3-39t-91$, 
$$
A=\left(\begin{matrix}
0 & 1 & 0 \\
0 & 0 & 1 \\
91 & 39 & 0
\end{matrix}
\right),\quad
B=\left(\begin{matrix}
7 & 0 & 0 \\
5 & 1 & 0 \\
-24 & -4 & 1
\end{matrix}
\right),\quad
C=\left(\begin{matrix}
49 & 0 & 0 \\
33 & 1 & 0 \\
4 & -4 & 1
\end{matrix}
\right).
$$
Since $\det A=\det B=\det C$, by Lemma \ref{l:orbit}, $\Z^d$-actions $Y_A, Y_B,Y_C$ are orbit equivalent.
All three $A$, $B$, and $C$ are conjugate to each other in $\GL_3(\Q)$. Moreover, $A$ is a companion matrix of $B$ and $C$. The three matrices above give (all) three equivalence classes of integer matrices with characteristic polynomial $h$ up to conjugation by elements in $\GL_3(\Z)$, {\it i.e.}, any matrix in $\M_3(\Z)$ with characteristic polynomial $h$ is 
$\GL_3(\Z)$-conjugate to $A$, $B$, or $C$, and any two matrices out of $A$, $B$, and $C$ are not $\GL_3(\Z)$-conjugate to each other (this can be verified using \cite{db}, \cite{sage}). Using the  methods of \cite{s1}, we can prove that any two out of $Y_A,Y_B$, and $Y_C$ are not continuously orbit equivalent.
\end{example}

\begin{example}\label{ex:6} 
In this example, $d=4$, $A,B\in\M_4(\Z)$ have the same irreducible characteristic polynomial $h(t)=t^4+t^2+9$, 
$$
A=\left(\begin{matrix}
0 & 1 & 0 & 0 \\
0 & 0 & 1 & 0 \\
0 & 0 & 0 & 1 \\
-9 & 0 & -1 & 0
\end{matrix}
\right),\quad
B=\left(\begin{matrix}
0 & 1 & -1 & 0 \\
9 & 0 & 2 & 1 \\
9 & 0 & 1 & 1 \\
-18 & -9 & 7 & -1
\end{matrix}
\right).
$$
One can show that Lemma \ref{l:conjs}$(i)$ holds for $A$, $B$, so that 
$G_A=G_B$ (see also \cite[Example 11]{s1}). Thus, $\Z^d$-actions $Y_A, Y_B$ are conjugate.
\end{example}

\end{document}